\documentclass[11pt]{article}
\usepackage{amsthm, amsmath, amssymb, amsfonts, mathtools, url, booktabs, tikz, setspace, fancyhdr, bm}
\usepackage{geometry}
\usepackage{graphicx}
\usepackage{xcolor}
\usepackage{comment}
\usepackage[pdftex,colorlinks,backref=page,citecolor=blue,bookmarks=false]{hyperref}

\usepackage[square,sort,comma,numbers]{natbib} 
\date{}

\usepackage{enumerate}
\usepackage[shortlabels]{enumitem}
\usepackage[babel]{microtype}
\usepackage[english]{babel}
\usepackage[capitalise]{cleveref}
\usepackage{comment}
\usepackage{bbm}
\usepackage{csquotes}
\usepackage{mathabx}
\usepackage{tikz}

\counterwithin{figure}{section}

\newtheorem{theorem}{Theorem}[section]

\newtheorem{obs}[theorem]{Observation}
\newtheorem{lemma}[theorem]{Lemma}
\newtheorem{corollary}[theorem]{Corollary}

\newtheorem{claim}[theorem]{Claim}

\theoremstyle{definition}
\newtheorem{definition}[theorem]{Definition}

\newcommand{\lam}{\lambda}
\newcommand{\one}{\mathbbm{1}}

\title{Cycles with many chords}
\author{
Nemanja Dragani\'c\thanks{
Department of Mathematics, ETH, Z\"urich, Switzerland. Research supported in part by SNSF grant 200021\_196965.
\textbf{\{nemanja.draganic, abhishek.methuku, david.munhacanascorreia, benjamin.sudakov\}@math.ethz.ch}.
}
\and
Abhishek Methuku\footnotemark[1]
\and
David Munh\'a Correia\footnotemark[1]
\and
Benny Sudakov\footnotemark[1]}

\begin{document}
\maketitle

\begin{abstract}
How many edges in an $n$-vertex graph will force the existence of a cycle with as many chords
as it has vertices? Almost 30 years ago, Chen, Erd\H{o}s and Staton considered this question and showed that any $n$-vertex graph with $2n^{3/2}$ edges contains such a cycle. 
We significantly improve this old bound by showing that $\Omega(n\log^8n)$ edges are enough to guarantee the existence of such a cycle. 

Our proof exploits a delicate interplay between certain properties of random walks in almost regular expanders. We argue that while the probability that a random walk of certain length in an almost regular expander is self-avoiding is very small, one can still guarantee that it spans many edges (and that it can be closed into a cycle) with large enough probability to ensure that these two events happen simultaneously.
\end{abstract}

\section{Introduction}
One of the classical problem frameworks in combinatorics deals with questions of the following type. How many edges does an $n$-vertex graph need to have to contain a subgraph with a certain prescribed structure? In many instances of such problems, it turns out that we can find subgraphs with very interesting structure only assuming very weak bounds on the number of edges.

For example, Janzer and Sudakov~\cite{janzer2022resolution} showed that any $n$-vertex graph with average degree at least $\Omega(\log\log n
)$ contains a $k$-regular subgraph, which is optimal up to a constant factor and answers an old question of Erd\H{o}s and Sauer. Liu and Montgomery~\cite{liu2023solution} recently solved several open problems using methods related to sublinear expansion. In particular, they showed that any graph with a large enough constant average degree contains a cycle whose length is a power of $2$.
Another result of similar flavour by Buci{\'c}, Gishboliner, and Sudakov~\cite{bucic2022cycles} shows that for $k\geq 3$, every $k$-regular Hamiltonian graph has cycles of $n^{1-o(1)}$ many lengths, asymptotically solving a problem of Jacobson and Lehel. Furthermore, Fernández and Liu~\cite{fernandez2023build} proved a conjecture of Thomassen \cite{thomassen1989configurations}, showing that large enough constant average degree forces the existence of a pillar (two vertex-disjoint cycles of the same length, along with vertex-disjoint paths of the same length which connect matching vertices in order around the cycles).

Many of the problems of this sort also deal with conditions which force the existence of cycles with chords. Answering a question of Erd\H{o}s~\cite{erdHos1975problems}, Bollob\'as~\cite{MR539938} proved that a large enough constant average degree is enough to force the existence of a cycle whose chords also contain a cycle. Extending this result, Chen, Erd\H{o}s and Staton~\cite{chen1996proof} proved that for every $k\geq 2$ there is a constant $c_k$ such that any graph with average degree at least $c_k$ contains $k$ cycles $C_1,\ldots, C_k$, such that the edges of $C_{i+1}$ are chords of the cycle $C_i$. This answered a question of Bollob\'as~\cite{MR539938}. More recently, Fern{\'a}ndez, Kim, Kim and Liu~\cite{fernandez2022nested} strengthened the result of Bollob\'as, showing that large enough constant average degree is enough to force the existence of a cycle whose chords contain a cycle whose vertices follow the orientation of the first cycle. Another similar result was shown by Thomassen \cite{thomassen1983girth}, who proved that for every $k\geq 1$, there exists $g_k$ such that any graph with minimum degree $3$ and girth at least $g_k$ contains a cycle with at least $k$ chords. 

In 1996, Chen, Erd\H{o}s and Staton~\cite{chen1996proof} also considered the following natural question: how many edges force the existence of a cycle with as many chords as it has vertices? They showed that if an $n$-vertex graph has minimum degree at least $2\sqrt{n}$ then it contains a cycle which has $n$ chords, thus showing that $2n^{3/2}$ edges are enough. 
In this paper, we significantly improve this old result of Chen, Erd\H{o}s and Staton, by showing that $\Omega(n\log^8n)$ edges are enough to force a cycle with at least as many chords as it has vertices.

\begin{theorem}\label{thm:main}
If $n$ is sufficiently large, then every $n$-vertex graph with at least $n\log ^8 n$ edges contains a cycle $C$ with at least $|C|$ chords.
\end{theorem}

\noindent \textbf{Overview of the proof.}  Initially, we undertake a process of cleaning the graph, i.e. finding a subgraph that is nearly regular (with a constant factor difference between the minimum and maximum degrees), good expansion properties, and a sufficiently high average degree. Subsequently, we investigate a random walk of an appropriate length $\ell$ within this subgraph. We consider two critical events: firstly, we analyze the probability that the random walk is self-avoiding, meaning that it does not revisit any of the previously visited vertices. Secondly, we assess what is the likelihood of the set of vertices visited by the random walk to span at least $\ell$ chords. While the occurrence of the first event is characterized by an exponentially small probability $q$, we carefully establish that the second event still holds with probability more than $1-q$.  Crucially, for bounding the probability for the first event we use the fact that the obtained graph has good expansion properties. For the second event, directly applying standard concentration inequalities does not yield a strong enough bound on the required probability. To remedy the situation, we prove an edge-decomposition result in almost-regular graphs, which combined with concentration inequalities produces the required bound. \\

\noindent \textbf{Notation.} We use standard graph theoretic notation throughout the paper. In particular, for a graph $G$, we denote by $d(G)$ its average degree, and by $\delta(G)$, $\Delta(G)$ its minimum degree and maximum degree, respectively. By $e(G)$, we denote the number of edges of $G$, and for $S\subseteq V(G)$, by $e_G(S)$ we denote the number of edges induced by $S$. For two disjoint sets $A, B \subseteq V(G)$, $e_G(A, B)$ is the number of edges of $G$ which are incident to both $A$ and $B$. We omit the subscripts if it is clear from the context which graph we refer to. Given an event $E$ in a probability space, we denote by $\one_E$ the indicator random variable of $E$, which is equal to $1$ when $E$ holds, and $0$ otherwise.

\section{Preliminaries}
In this section, we collect several useful definitions and results used in our proofs.

\begin{definition}
Let $K > 0$ and let $G$ be a graph. We say that $G$ is $K$-almost-regular if $\Delta(G) \le K \delta(G)$.
\end{definition}

\begin{definition}
Let $\lambda > 0$. We say that a graph $G$ is a $\lambda$-expander if every set $X \subseteq V(G)$ with $|X| \le \frac{1}{2}|V(G)|$ satisfies $e(X, \overline{X}) \ge \lambda d(G) |X|$.   
\end{definition}

\noindent We use the following lemma from~\cite{bucic2020nearly}, which states that every graph contains an almost regular subgraph whose average degree is at most by a logarithmic factor smaller than that of the original graph.
\begin{lemma} \label{almostregsubgraph} Every graph $G$ on $n$ vertices contains a $6$-almost regular subgraph $G' \subseteq G$ with average degree at least $\frac{d(G)}{100 \log n}$.
\end{lemma}

\subsection{Finding an almost regular expanding subgraph}
\noindent The goal of this subsection is to prove the following standard statement which allows us to find a (weakly) expanding subgraph in any graph with large enough degree. Its proof is a standard application of the density increment method.

\begin{lemma}
\label{findingalmostregexpander}
Let $G$ be an $n$-vertex graph with average degree $d \ge \log^2 n$ and let $n$ be large enough. Then there exists a bipartite subgraph $G' \subseteq G$ with the following properties.
\begin{itemize}
    \item $d(G') \geq \frac{d}{600\log n}$.
    \item $G'$ is $100$-almost-regular.
    \item $G'$ is a $\frac{1}{10\log n}$-expander.
\end{itemize}
\end{lemma}
\begin{proof}
First, let $G_0$ be a bipartite subgraph $G$ with average degree at least $\frac{d}{3}$.
Then, we apply Lemma~\ref{almostregsubgraph} to $G_0$ in order to find a (bipartite) $6$-almost regular subgraph $G_1 \subseteq G_0$ with $d(G_1) \geq \frac{d}{300 \log n}$. Now let $\lambda \coloneqq \frac{1}{2 \log n}$ and let $d_1 := d(G_1)$.

We now perform a procedure which finds the desired subgraph $G'$ in $G_1$. At every step, we consider a subgraph $H$ and show that either $G' := H$ satisfies the desired properties and we finish the procedure or we will find a certain subgraph $H' \subseteq H$ and continue the procedure with $H'$. We will then show that at some point this procedure must finish. 

Let us now describe a step in this procedure. Consider a subgraph $H$ with average degree $d(H)$. If $H$ has a vertex $v$ with degree less than $d(H)/2$ we remove it, and define $H' := H \setminus v$ and proceed to the next step with $H'$. Note that $H'$ has average degree at least $d(H)$.

\begin{claim}
\label{claimexpcleaning}
If there is a set $U \subseteq V(H)$ with $|U| \leq \frac{|V(H)|}{2}$ such that $e(U,\bar{U}) < \frac{\lambda}{3} d(H)|U|$, then we have either $d(H[\overline{U}]) \geq d(H)$ or $d(H[U]) \geq (1-\lambda)d(H)$. 
\end{claim}
\begin{proof}
Suppose otherwise. Then by the assumption on $e(U,\bar{U})$ we have $e(H) \leq e(H[U]) + e(H[U,\overline{U}]) + e(H[\overline{U}])  < |U| d(H) \left( \frac{1-\lambda}{2} + \frac{\lambda}{3} \right) + |\overline{U}| \frac{d(H)}{2} < |V(H)| \frac{d(H)}{2}$ which is a contradiction.   
\end{proof}
\noindent By Claim~\ref{claimexpcleaning}, if there is any set $U \subseteq V(H)$ with $|U| \leq \frac{|V(H)|}{2}$ such that $e(U,\bar{U}) < \frac{\lambda}{3} d(H)|U|$ then either $d(H[\overline{U}]) \geq d(H)$ or $d(H[U]) \geq (1-\lambda)d(H)$.  If $d(H[\overline{U}]) \geq d(H)$, we define $H' := H[\overline{U}]$ and proceed to the next step with $H'$. On the other hand, if $d(H[U]) \geq (1-\lambda)d(H)$, we define $H' := H[U]$ and proceed to the next step with $H'$. 

Now, note that at any step where the procedure does not terminate, the following always holds: if $|V(H')| \geq |V(H)|/2$ then $d(H') \geq d(H)$; if $|V(H')| < |V(H)|/2$ then $d(H') \geq (1-\lambda)d(H)$. Also, we have $|V(H')| < |V
(H)|$. Furthermore, since $|V(H')| < |V(H)|/2$ can only occur for at most $\log n$ steps, at any step the subgraph $H$ we consider satisfies $d(H) \geq (1 - \lambda \log n)d(G_1) \ge d(G_1)/2 > 0$, and therefore, the procedure must eventually stop with a non-empty subgraph $G'$ and $d(G') \ge d(G_1)/2$. 

We now show that this final subgraph $G'$ satifies the desired properties. Firstly, the discussion above implies that $d(G') \geq \frac{d(G_1)}{2} \geq \frac{d}{600 \log n}$. Secondly, note that since the procedure removes every vertex of low degree, we have $\delta(G') \geq d(G')/2$. Since $G_1$ was $6$-almost regular and $d(G') \geq  d(G_1)/2$ we have that $\Delta(G') \leq 100 \delta(G')$ as desired. Finally, the procedure also implies that every set $U \subseteq V(G')$ of size at most $|V(G')|/2$ satisfies $e(U,\bar{U}) \ge \frac{\lambda}{3} d(G')|U| \ge \frac{1}{10 \log n}d(G')|U|$, so $G'$ is a $\frac{1}{10 \log n}$-expander, as required.
\end{proof}

\subsection{Random walks in expanders}
In this subsection we compute the mixing time of a random walk in an almost regular expander. The notation and results that we cite in this subsection can be found in~\cite{jiang2021rainbow} and~\cite{lovasz1993random}.
    
Let $G$ be a connected graph on the vertex set $[n]$. 
    Consider a random walk on $V(G)$, where we start at some vertex $v_0$ and at the $i$-th step we move from $v_i$ to one of its neighbours, say $v_{i+1}$, where
    each neighbour of $v_i$ is chosen as $v_{i+1}$ with probability $\frac{1}{d(v_i)}$.  Let $M$ be an $n \times n$ matrix defined as follows. Let $M_{v, u}$ be the probability of stepping from $v$ to $u$; so $M_{v, u}= \frac{1}{d(v)}$ if $vu \in E(G)$, and $M_{v, u} = 0$ otherwise. Denote by $D$ the $n \times n$ diagonal matrix with $D_{v, v} = \frac{1}{d(v)}$ for $v \in [n]$, and let $A$ be the adjacency matrix of $G$. Then $M=D A$. 
    So the probability that a random walk starting at vertex $v$ ends in $u$ after $t$ steps is $(M^t)_{v, u}$.

    \begin{definition}
    \label{def19}
        Let the graph $G$ and matrices $M,D,A$ be as above and define $N(G) =D^{1/2}AD^{1/2}$.
        Note that the matrix $N(G)$ is symmetric, so it has $n$ real eigenvalues. Let $\lam_1(N)\geq \lam_2(N)\geq \dots \geq \lam_n(N)$ denote the eigenvalues of $N:=N(G)$. 
    \end{definition}

    \begin{lemma}[Lemma 5.2 in \cite{jiang2021rainbow}] 
    \label{lem:markovchain}
        Let $G$ be a connected $n$-vertex bipartite graph, with the bipartition $\{X, Y\}$ with $m$ edges. 
        Let $M=D(G)A(G)$ and $N = N(G)$.
        Then for every $v,u \in V(G)$ and integer $k \ge 1$, we have
        \begin{equation*}
            \left|(M^k)_{v, u} - \frac{d(u)}{2m}\left(1 + (-1)^{k + \one_{v \in X} + \one_{u \in X}}\right) \right|
            \leq \sqrt{\frac{d(u)}{d(v)}} \cdot \big(\lam_2(N)\big)^k.
        \end{equation*}
    \end{lemma}
    
\noindent Note that Lemma~\ref{lem:markovchain} says that when $k$ is even and both $v,u$ are in the same part or when $k$ is odd and $v,u$ are in different parts then   $ \left|(M^k)_{v, u} - \frac{d(u)}{m} \right|
            \leq \sqrt{\frac{d(u)}{d(v)}} \cdot \big(\lam_2(N)\big)^k.$ Also observe that when $k$ is even and $v$ and $u$ are in different parts or when $k$ is odd and $v$ and $u$ are in the same part then $(M^k)_{v, u}=0$. 

 \begin{definition}[Conductance] \label{def:conductance}
        For a graph $G$ with $m$ edges, let $\pi(v)= \frac{d(v)}{2m}$, and for any $S \subseteq V(G)$,  let $\pi(S) \coloneqq \sum_{s\in S}{\pi(s)}$; observe that $\pi(S) \le 1$ for every $S \subseteq V(G)$. Define the \emph{conductance} of a set $S$, denoted by $\Phi(S)$, as
        \begin{equation*}
            \Phi(S) \coloneqq \frac{e(S,\overline{S})}{2m \cdot \pi(S)\pi(\overline{S})}, 
        \end{equation*}
        and let the \emph{conductance} of a graph $G$, denoted by $\Phi_G$, be defined as
        \begin{equation*}
            \qquad \Phi_G \coloneqq \min_{S \subseteq V(G)} \Phi(S).
        \end{equation*}
    \end{definition}
   
    \begin{theorem}[Theorem 5.3 in \cite{lovasz1993random}] 
\label{Lovasz}\label{thm:upperboundoneigenvalue} 
        Let $G$ be a graph and let $\lam_2 = \lam_2(N(G))$. Then $\lam_2 \le 1 - \frac{\Phi_G^2}{8}$.
    \end{theorem}

\begin{lemma}
\label{boundingconductance}
  Let $\lambda > 0$, $K \ge 1$, and let $G$ be a $K$-almost-regular $\lambda$-expander. Then $\Phi_G \ge \frac{\lambda}{K}$.
\end{lemma}
\begin{proof}
  Suppose all the vertices in $G$ have their degrees between $d$ and $K d$.
  Let $S\subseteq V(G)$. Since $\Phi(S)=\Phi(\overline{S})$, we may assume that
        $|S|\leq \frac{n}{2}$. Since $G$ is a $\lambda$-expander, we have 
        $e(S, \overline{S})\geq \lambda d(G) |S| \ge \lambda d |S|$.
        Note that $\sum_{v\in S} d(v) \le K d |S|$. 
        Also, observe that $\pi(\overline{S})\leq 1$. Hence
   
        \begin{align*}
            \Phi(S)& =\frac{e(S, \overline{S})}{2e(G)\pi(S)\pi(\overline{S})} 
            \geq \frac{e(S,\overline{S})}{\sum_{v\in S} d(v)}
            \geq  \frac{ \lambda d|S| }{K d|S|}\geq \frac{\lambda}{K}.
        \end{align*}
        The above inequality thus implies $\Phi_G \ge \frac{\lambda}{K}$.
\end{proof}

\noindent Combining Lemma~\ref{boundingconductance} and Theorem~\ref{Lovasz}, we obtain that if $G$ is a $K$-almost-regular $\lambda$-expander and $\lam_2 = \lam_2(N(G))$, then $\lambda_2 \le 1 - \frac{1}{8}(\frac{\lambda}{K})^2$. Therefore, Lemma~\ref{lem:markovchain} implies the following. 

\begin{corollary}\label{cor:mixingtime}
     Let $\lambda > 0$, $K \ge 1$, and let $G$ be a bipartite graph on $n$ vertices which is a $K$-almost-regular $\lambda$-expander. Let $\{X, Y\}$ be the bipartition of $G$ with $m$ edges and no isolated vertices. 
        Let $M=D(G)A(G)$ and $N = N(G)$.
        Then for every $v,u \in V(G)$ and integer $k \ge 1$, the probability $(M^k)_{v, u}$ that a random walk starting at vertex $v$ ends in $u$ after $k$ steps satisfies
        \begin{equation*}
            \left|(M^k)_{v, u} - \frac{d(u)}{2m}\left(1 + (-1)^{k + \one_{v \in X} + \one_{u \in X}}\right) \right|
            \leq \sqrt{K} \cdot \left(1 - \frac{1}{8}\left(\frac{\lambda}{K}\right)^2\right)^k.
        \end{equation*}
\end{corollary}
\noindent
We shall utilize the following definition of \emph{mixing time} in bipartite graphs.
\begin{definition}[Mixing time]
    Let $G$ be a bipartite graph on $n$ vertices.
    Let $\{X, Y\}$ be the bipartition of $G$ with $m$ edges and no isolated vertices. We say that $G$ has \emph{mixing time $k$} if
        for every $u,v \in G$, the probability $(M^k)_{v, u}$ that a random walk starting at vertex $v$ ends in $u$ after $k$ steps satisfies
        \begin{equation*}
            \left|(M^k)_{v, u} - \frac{d(u)}{2m}\left(1 + (-1)^{k + \one_{v \in X} + \one_{u \in X}}\right) \right|
            \leq \frac{1}{n^2}.
        \end{equation*}
\end{definition}
\noindent The following is a corollary of the previous statements and succinctly summarizes a few pertinent properties of mixing time that are essential for our proofs.
\begin{corollary}\label{cor:mixing time for sets}
Let $K \ge 1$, let $G$ be a connected $K$-almost-regular bipartite graph on $n$ vertices with mixing time $k$ and suppose $n$ is large enough. Let $\{X, Y\}$ be the bipartition of $G$. Then the following holds:
\begin{itemize}
    \item[$\mathrm{(i)}$] 
For any given set $S \subseteq V(G)$, the probability that a random walk starting at a given vertex ends in a vertex of $S$ after at least $k$ steps is at most $\frac{4K}{n}|S|$.
\item [$\mathrm{(ii)}$]  If $k'$ is even (if $k'$ is odd), then the probability that a random walk starting at a vertex of $X$ ends in any given vertex of $X$ (of $Y$ respectively) after $k' \ge k$ steps is at least $\frac{1}{K n}.$ 
\item [$\mathrm{(iii)}$] 
If $G$ is a $\lambda$-expander for some $\lambda > 0$, then it has mixing time $k \leq \frac{30 K^2}{\lambda^2}  \log n$.
\end{itemize}
\end{corollary}
\begin{proof}
 Suppose $G$ has $m$ edges. Since $G$ is $K$-almost-regular and $k$ is its mixing time, the probability that a random walk starting at a given vertex ends in a vertex of $S$ after $k$ steps is at most $$\sum_{u \in S} \left(\frac{d(u)}{m} + \frac{1}{n^2}\right) \le \frac{4K}{n} |S|.$$
Similarly, the required probability in $\mathrm{(ii)}$ is at least $$\frac{d(u)}{m} - \frac{1}{n^2} \ge \frac{1}{K n}.$$
Finally, by Corollary~\ref{cor:mixingtime}, if $G$ is a $\lambda$-expander, then it has mixing time at most $\frac{30 K^2}{\lambda^2} \log n$.
\end{proof}

\section{Proof}
As mentioned earlier, our strategy for proving Theorem \ref{thm:main} is to first pass to an almost-regular expander (using Lemma~\ref{findingalmostregexpander}).

In Section~\ref{sec:findingstarforests}, we show that one can find a collection of star-forests in almost-regular graphs. In Section~\ref{sec:randomwalkonsets}, we prove a concentration inequality that allows us to show that a random walk must contain many vertices from any large enough set with high probability. Using this result and the star-forests that we found, we show that the random walk must contain many chords and that it can be closed into a cycle with high enough probability in Section~\ref{sec:chordsinwalks}. In Section~\ref{sec:self-avodingwalks}, we compute the probability that a random walk is self-avoiding, and we put everything together and prove Theorem~\ref{thm:main} in Section~\ref{sec:puttingeverythingtogether}.

\subsection{Finding star forests in an almost-regular graph}
\label{sec:findingstarforests}
\noindent Given disjoint sets $A$ and $B$, an \emph{$AB$-star-forest} $F$ is a set of vertex-disjoint stars such that the root of each of the stars is in $A$ and the leaves are in $B$. Two star forests $F$ and $F'$ are called \emph{root-disjoint} if the set of root vertices of the stars in $F$ is disjoint from the set of root vertices of the stars in $F'$. 

\begin{lemma}
\label{first_star_forest}
Let $G = (A,B)$ be a 1000-almost-regular bipartite graph on $n$ vertices, and let $d \coloneqq \delta(G) \ge 10^6$. Then, there exists an $AB$-star-forest $F \subseteq G$ consisting of $\frac{n}{100 d}$ stars of size $\frac{d}{10^6}$. 
\end{lemma}
\begin{proof}
 Note that since $G$ is 1000-almost-regular, by double counting the edges we have 
$|B| \leq 1000|A|$. Then $n = |A| + |B| \le 1001 |A|$, and so $|A| \ge n/1001$. Let $F$ be a maximal $AB$-star-forest $F \subseteq G$ consisting of stars of size $d/10^6$ and for the sake of contradiction assume that $F$ contains less than $\frac{n}{100d}$ stars. Consider $A' := A \setminus V(F), A'' := A \cap V(F)$ and $B':= B \setminus V(F)$, $B'':= B \cap V(F)$. By assumption, note that $|A''|<n/100d \le n/10^5 \le |A|/4$ and $|B''|\leq n/10^8$. Every vertex in $A'$ must have less than $d/10^6$ neighbours in $B'$, as otherwise $F$ would not be maximal.
Hence $e(A',B'')\geq |A'| (1-10^{-6})d \ge |A|d/4$. Therefore there is a vertex in $B''$ with degree at least $$\frac{e(A',B'')}{|B''|}\geq \frac{|A|d}{4|B''|}\geq  \frac{|A|d}{4n/10^8}> 1000 d$$ 
a contradiction with the fact that $G$ is 1000-almost-regular.
\end{proof}
\noindent Repeated application of the lemma above produces a collection of root-disjoint $AB$-star forests. This is shown by the following corollary.
\begin{corollary}\label{cor:star forests}
Let $G = (A,B)$ be a 100-almost-regular bipartite graph on $n$ vertices, and let $d:=\delta(G) \ge 10^7$. Then, there exist $\frac{d}{10}$ root-disjoint $AB$-star-forests $F_1, F_2, \ldots, F_{d/10}$ in $G$, where each $F_i$ consists of $\frac{n}{10^5 d}$ stars of size $\frac{d}{10^7}$.
\end{corollary}
\begin{proof}
Suppose we have already found the desired $AB$-star-forests $F_1, F_2, \ldots, F_{i}$ for $i < \frac{d}{10}$, we find the next one $F_{i+1}$ as follows. We remove the root vertices of the stars in $F_1, F_2, \ldots, F_{i}$ from $A \subseteq V(G)$.
This removes at most $i \cdot \frac{n}{10^5 d} \cdot 100 d < \frac{nd}{10^4}$ edges from $G$ as $\Delta(G) \le 100d$ (since $G$ is 100-almost-regular), let the resulting graph be $G'$. Hence, $G'$ still has at least $(\frac{1}{2} - 10^{-4}) nd$ edges, so by repeatedly removing vertices of degree less than $\frac{d}{4}$, we obtain a subgraph $G''$ of $G'$ with minimum degree $\frac{d}{4} \ge 10^6$, while its maximum degree is still at most $100 d$, so $G''$ is 1000-almost-regular. Moreover, $G''$ contains at least $(\frac{1}{2} - 10^{-4}) nd-\frac{nd}{4} \ge (\frac{1}{4} - 10^{-4})nd$ edges, so $G''$ has at least $\frac{2(\frac{1}{4} - 10^{-4})nd}{100 d} \ge \frac{n}{250}$ vertices (as $\Delta(G'') \le 100 d$). So by Lemma~\ref{first_star_forest}, $G''$ has an $AB$-star-forest $F_{i+1}$ consisting of $\frac{n}{10^5 d}$ stars of size $\frac{d}{10^7}$, as desired. 
\end{proof}

\subsection{Intersection of random walks with arbitrary sets}
\label{sec:randomwalkonsets}
\noindent
For a random walk $W=\{X_i\}$ on a graph $G$ with mixing time $k$, the set of vertices $\{X_{ik}: i\in [t]\}$ for some $t\geq 1$ behaves almost like a random set of $t$ vertices chosen uniformly at random with repetition from $G$. We exploit this fact in this subsection.

More precisely, let $G$ be a $100$-almost-regular bipartite graph on $n$ vertices with parts $A,B$ and consider a random walk $R$ of length $t$ starting at some vertex $v_0 \in A$ and take a $k' \in \{k,k+1\}$ which is odd, where $k$ is the mixing time of $G$. Let $S$ be a random set obtained by the following procedure which consists of $\lfloor t/k' \rfloor$ steps:  
\begin{itemize}
    \item In each step $1 \leq i \leq \lfloor t/k' \rfloor$, with probability $10^{-5}$  we either choose a uniformly random vertex $v_i$ from $A$ (if $i$ is even) or from $B$ (if $i$ is odd), or we do nothing (with probability $1 - 10^{-5}$).
    \item The set $S = \{v_i : 1 \leq i \leq \lfloor t/k' \rfloor\}$ consists of all of the chosen vertices.
\end{itemize}
\noindent 
Now, consider the set of vertices $R(k) \coloneqq \{u_1,u_2, \ldots \}$ where for each $1 \leq i \leq \lfloor t/k' \rfloor$, $u_i$ is the $ik'$-th vertex of the random walk $R$. We then have the following property given by the definition of mixing time and Corollary~\ref{cor:mixing time for sets}.
\begin{obs}    
    Conditioning on the choice of $u_0,u_1, \ldots, u_{i-1}$,
    we have that for every $a\in A$ and $b\in B$ it holds that:
    \begin{itemize}
    \item  If $i$ is even, then $\mathbf P[u_i=a]\geq \frac{1}{100 n} \ge \frac{10^{-5}}{|A|}$ (since $n = |A| + |B| \le 101 |A|$ as $G$ is 100-almost-regular)
    \item If $i$ is odd, then $\mathbf P[u_i=b]\geq \frac{10^{-5}}{|B|}$. 
    \end{itemize}
\end{obs}

By definition of $S$ this implies that for any fixed set $X \subseteq V(G)$, the random variable $|R(k) \cap X|$ stochastically dominates $|S \cap X|$. Therefore, we have the following.
\begin{lemma}\label{lem:randomwalkintersection}
Let $G$ be an $n$-vertex $100$-almost regular bipartite graph with parts $A,B$ with mixing time $k$ and let $R$ be a random walk in $G$ of length $t\leq 10 n$ starting at a given vertex. Then, for any set $X \subseteq V(G)$ we have that
$$\mathbb{P} \left(|R(k) \cap X| \leq \frac{|X|t}{10^9 kn} \right) \le e^{-\frac{|X|t}{10^9 kn}}.$$
\end{lemma}
\begin{proof}
Let $R= \{X_i : 0 \le i \le t \}$ be a random walk, and without loss of generality, suppose $X_0 \in A$. 
As noticed before, $|S \cap X|$ is stochastically dominated by $|R(k) \cap X|$, so it is enough to show the statement with $|S \cap X|$ instead of $|R(k) \cap X|$.
Note that either $X \cap A$ or $X \cap B$ has size at least $|X|/2$. Suppose without loss of generality that $|X \cap A| \ge |X|/2$, the other case is very similar. Let $C:= 10^{-9}\frac{|X|t}{kn}$, and consider the procedure that was used to define $S$, where in each step $1 \le j \le \lfloor t/k' \rfloor$, a vertex $v_j$ is (randomly) added to $S$. Suppose that a new vertex from $X$ is added to $S$ in only at most $i \le C$ steps; we are interested in the probability that this event occurs. Fix such a choice of $i$ steps. In any such step, the probability that a new vertex from $X$ is added to $S$ is at most $\max\{10^{-5} \frac{|X|}{|A|}, 10^{-5} \frac{|X|}{|B|} \} \le \frac{|X|}{100 n}$ (since $G$ is 100-almost-regular). Moreover, note that since at most $C \le \frac{|X|}{4}$ vertices from $X$ have been added to $S$, the probability that a new vertex from $X \cap A$ is added to $S$ in any step $j$ (with $j$ even) is at least $10^{-5} \frac{|X|/2 - |C|}{|A|} \ge 10^{-6} \frac{|X|}{n}$. Therefore, as there are at least $\frac{1}{2}\lfloor t/k' \rfloor - C \ge \frac{1}{4}\lfloor t/k' \rfloor$ steps $j$ (with $j$ even) where no new vertex is added to $S$, the required probability is at most 
\begin{align*}
\sum_{0 \leq i \leq C}  {\lfloor t/k' \rfloor \choose i} \left(\frac{|X|}{100 n}\right)^{i} \left(1 - 10^{-6}\frac{|X|}{ n}\right)^{\frac{1}{4}\lfloor t/k' \rfloor}. 
\end{align*}

\noindent Since the common ratio satisfies $\frac {{\lfloor t/k' \rfloor \choose i} \left(\frac{|X|}{100 n}\right)^{i}}{{\lfloor t/k' \rfloor \choose i-1} \left(\frac{|X|}{100 n}\right)^{i-1}} \ge \frac{\lfloor t/k' \rfloor - (i-1)}{i} \frac{|X|}{100 n} \ge \frac{\lfloor t/k' \rfloor - C}{C} \frac{|X|}{100 n} \ge \frac{\lfloor t/k' \rfloor }{2 C} \frac{|X|}{100 n} \ge 2$ for $0 \le i \le C$, the above sum is at most
\begin{align*}
2{\lfloor t/k' \rfloor \choose C} \left(\frac{|X|}{100 n}\right)^{C} \left( 1 - 10^{-6}\frac{|X|}{n}\right)^{ \frac{t}{8k'}}
\le
\left(\frac{2 e (t/k') |X| }{C \cdot 100 n }\right)^C \left( 1 - 10^{-6}\frac{|X|}{n}\right)^{ \frac{t}{8k'}} \leq e^{-\frac{|X|t}{10^9 kn}},
\end{align*}

\noindent where in the last inequality we used $\left(\frac{2 e (t/k') |X| }{C \cdot 100 n }\right)^C \le 10^{10 C} \le e^{30 C} \le e^{10^{-7}\frac{|X|t}{kn}}$ by plugging in the value of $C$, and that $\left( 1 - 10^{-6}\frac{|X|}{n}\right)^{ \frac{t}{8k'}} \le e^{-10^{-6}\frac{|X|}{n} \cdot \frac{t}{8k'}} \le e^{-10^{-6}\frac{|X|}{n} \cdot \frac{t}{9 k}}$. This completes the proof of the lemma.
\end{proof}

\subsection{Chords in random walks}
\label{sec:chordsinwalks}

\noindent In this subsection, we show that with very high probability, the graph induced by the vertices of two random walks contains many edges.
Recall that for a random walk $R$ of length $t$, we denote $R(k) = \{u_1,u_2, \ldots \}$ where for each $1 \leq i \leq \lfloor t/k' \rfloor$, $u_i$ is the $ik'$-th vertex of the random walk $R$ and $k' \in \{k,k+1\}$ is odd.

\begin{lemma}\label{lem:many edges}
Let $G$ be a 100-almost-regular bipartite graph on $n$ vertices with $\delta(G) = d \ge 10^8$ and mixing time $k$. Let $R_1$ and $R_2$ be random walks in $G$ of length $t$ for $n\geq t \geq \max \{10^{17} \frac{kn}{d}, 10^{25} k \log n\}$, starting at arbitrary vertices $v_1$ and $v_2$, respectively. Then, 
$$\mathbf{P} \left(e(R_1(k), R_2(k)) \leq \frac{t^2 d}{10^{32} k^2 n } \right) \leq e^{-\frac{t}{10^{24} k}}.$$
\end{lemma}
\begin{proof}
Let $p := \frac{t}{10^9 k n}$, and note that by the bound on $t$ we have $p\geq\max\{\frac{10^8}{d},\frac{10^{16}\log n}{n}\}$. First, we apply \Cref{cor:star forests} to find $\frac{d}{10}$ root-disjoint $AB$-star-forests $F_i \subseteq G$, each consisting of $\frac{n}{10^5 d}$ stars of size $\frac{d}{10^7}$.

\begin{claim}\label{claim:}
For each $i$, with probability at least $1- e^{-np/10^{13}}$, there is a set of at least $\frac{n}{2\cdot 10^5 d}$ stars in $F_i$ such that $R_1(k)$ contains at least $\frac{dp}{10^7}$ leaves of each of those stars.
\end{claim}

\begin{proof}
Consider some $F_i$ and let $A,B$ be the bipartition of $G$, and denote by $m_i=\frac{n}{10^5 d}$ the number of stars in $F_i$.
Fix a collection of $m_i/2$ of stars in $F_i$ and note that the probability that every star in this collection has less than $dp/10^7$ leaves in $R_1(k)$ is at most $e^{-\frac{m_idp}{2\cdot10^7}}$ by applying \Cref{lem:randomwalkintersection} to the set of leaves of all the $m_i/2$ stars in the collection.

Hence, by the union bound over all such collections of $m_i/2$ stars of $F_i$, we have that the event from the statement of the claim does not hold with probability at most
$$2^{m_i} \cdot e^{-\frac{dpm_i}{2\cdot10^7}} \leq e^{\frac{dpm_i}{10^8}} \cdot e^{-\frac{dpm_i}{2\cdot10^7}} \leq e^{-\frac{dpm_i}{10^8}}=
e^{-\frac{np}{10^{13}}},$$
 where we used that $m_i \leq \frac{dpm_i}{10^8}$, since $p\geq 10^8/d$.
\end{proof}
\noindent
By a simple union bound and since $pn\geq 10^{16}\log n$, we then have that with probability at least $ 1 - ne^{-\frac{pn}{10^{13}}} \ge 1 - e^{-\frac{pn}{10^{14}}}$ the following holds: for every star-forest $F_i$, more than half of its stars each have at least $dp/10^7$ leaves in $R_1(k)$. Suppose this event occurs. Then for each $F_i$, let $A_i$ denote the set of vertices in $F_i$ which are the roots of stars with more than $dp/10^7$ leaves in $R_1(k)$. Then, we have $\sum_i |A_i| \geq \frac{1}{2} \cdot \frac{d}{10} \cdot \frac{n}{10^5 d} \ge \frac{n}{10^7}$ and so, by Lemma~\ref{lem:randomwalkintersection}, with probability at least $1-e^{-\frac{pn}{10^7}}$ we have that $|R_2(k) \cap \bigcup_i A_i| \geq \frac{pn}{10^7}$. Hence, by the choice of vertices in $\bigcup_i A_i$ we have $e(R_1(k), R_2(k)) \geq |R_2(k) \cap \bigcup_i A_i| \cdot \frac{dp}{10^7} \geq \frac{d p^2 n}{10^{14}}$ with probability at least $(1 - e^{-\frac{pn}{10^{14}}})(1-e^{-\frac{pn}{10^7}}) \ge 1 - e^{-\frac{np}{10^{15}}} = 1 - e^{-\frac{t}{10^{24} k}}$ as required. 
\end{proof}

\subsection{Self-avoiding walks in expanders}
\label{sec:self-avodingwalks}

In this subsection we show that a random walk with small mixing time in an almost-regular graph is self-avoiding with a certain positive probability.
The exact details are given in Theorem~\ref{thm:self-avoiding}, whose proof uses the ideas 
from \cite{pak2002mixing}, with the necessary changes to fit our setting.

Let $G$ be a graph with mixing time $k$. Denote by $\{X_t^v\}$ the nearest neighbour random walk in $G$ which starts at a vertex $v$.
For a vertex set $A \subseteq V(G)$, let $Q_t^v(A)$ denote the probability that $X_t^v\in A$, and let $E_A^v$ be the event that $X_t^v\notin A$ for all $t\in [k]$, i.e. the random walk starting at $v$ avoids the set $A$ in the first $k$ steps. 

\begin{theorem}\label{thm:self-avoiding}
    Let $\beta \coloneqq 10^{-28}$, and let $G$ be a 100-almost-regular graph with mixing time $k$ with $\delta(G) \ge \frac{10^3 k^2}{\beta}$. Then the probability that a random walk starting at any given vertex of $G$ and of length $\frac{\beta^2 n}{10^6 k}$ is self-avoiding is at least $e^{-\frac{\beta^3 n}{10^5 k^2}}$.
\end{theorem}
\begin{proof}
The following claim allows us to show that most vertices $v$ are such that if we start a random walk at the vertex $v$, it is likely to avoid a given set. 

\begin{claim}\label{cl:bad set}
    For every set $A \subseteq V(G)$, it holds that the set $B$ of vertices $v$ such that $\mathbf P(\overline{E_A^v})\geq \beta$ is of size at most $\frac{100 k|A|}{\beta}$.
\end{claim}
\begin{proof}[Proof of claim]
    Notice first that $$\mathbf P\Big(\overline{E_A^v}\Big)\leq \sum_{t=1}^k \mathbf P(X_t^v\in A)=\sum_{t=1}^k Q_t^v(A).$$

\noindent For every pair of vertices $v,u\in G$ we have that $Q_t^v(u)\leq 100 \cdot Q_t^u(v)$  because of our assumption that $G$ is 100-almost-regular. Indeed, for every walk $P=v_0,v_1,\ldots, v_t$ we know that 
    $$
    \mathbf P\big[X_i^{v_0}=v_i\text{  for all $i\in[t]$}\big]=\prod_{i=0}^{t-1} \frac{1}{d(v_i)}.
    $$
    Since $Q_t^v(u)=\sum_P \mathbf P \big[X_i^{v_0}=v_i\text{  for all } i\in[t]\big]$, where the sum is over all walks $P=v_0,v_1,\ldots, v_t$ of length $t$ with $v_0=v$ and $v_t=u$, we conclude that $\frac{Q_t^v(u)}{Q_t^u(v)}=\frac{d(u)}{d(v)}\leq \frac{\Delta(G)}{\delta(G)}\leq 100$.

\noindent Using this, we obtain the following:
\begin{align*}
\sum_{v\in[n]} \mathbf P\Big[\overline{E_A^v}\Big] &\leq \sum_{v\in[n]} \sum_{t=1}^k Q_t^v(A) = \sum_{v\in[n]} \sum_{t=1}^k \sum_{a\in A} Q_t^v(\{a\})\leq  \sum_{t=1}^k \sum_{a\in A} \sum_{v\in[n]} 100 Q_t^a(\{v\})= 
100 k|A|.
\end{align*}
This immediately implies the claim as $\sum_{v\in[n]} \mathbf P\Big[\overline{E_A^v}\Big] \ge \beta |B|$ by the definition of $B$.
\end{proof}
\noindent The following claim gives the probability that a random walk of length $k$ is self-avoiding and additionally avoids a fixed  set of $k$ vertices.
\begin{claim}
\label{selfavoidingkwalk}
Let $S\subseteq V(G)$ with $|S|\leq k$, and let $v\in V(G)$. The probability that $X_i^v\neq X_j^v$  for every $1\leq i<j\leq k$ and that $X_i^v\notin S$ for all $i\in[k]$, is at least $1-\frac{200 k^2}{\Delta(G)}$.   
\end{claim}
    \begin{proof}[Proof of claim]
        For each $i \le k$, out of the $d(X_i^v)$ neighbours of $X_i^v$, only at most $|S|+i$ vertices are contained in $S \cup \{X_0^v, X_1^v, \ldots, X_{i-1}^v\}$. Hence the required probability is at least 
\[
\prod_{i=0}^{k-1} \Big(\frac{d(X_i^v)-|S|-i}{d(X_i^v)}\Big)\geq \Big(1-\frac{2k}{\delta(G)}\Big)^k \geq 1-\frac{200 k^2}{\Delta(G)}.
\] \end{proof}

\noindent Now, for every $v\in V(G)$ and every set $A\subseteq V(G)$ define $$\alpha_v(A) \coloneqq\min_{S\subseteq V(G), |S|=k} \mathbf P \Big[ E_{A\cup S}^v\Big].$$

\begin{lemma}\label{lem:good set}
    Let $A\subseteq V(G)$ with $|A|\leq \frac{\beta^2n}{10^6 k}$, and let $X$ be the set of vertices $v\in V(G)$ such that $\alpha_v(A) \ge 1-\beta-\frac{200k^2}{\Delta(G)}$. Then $|X| \ge n-\frac{100 k|A|}{\beta}$. In particular, for every $v \in V(G)$ we have $Q_k^v(X)\geq 1- \frac{10^5 k|A|}{\beta n} \geq 1 - \frac{\beta}{10}$. 

\end{lemma}
\begin{proof}
Note that by Claim~\ref{selfavoidingkwalk}, we have $\alpha_v(A) \geq \mathbf P \Big[ E_{A}^v\Big]-\frac{200 k^2}{\Delta(G)}$. For every vertex $v \not \in X$, we have  $\alpha_v(A) \le 1-\beta-\frac{200k^2}{\Delta(G)}$, so $\mathbf P \Big[ E_{A}^v\Big] \le 1 - \beta$ i.e., $\mathbf P \Big[ \overline{E_{A}^v}\Big] \ge \beta$. Hence, by Claim~\ref{cl:bad set}, $|\overline{X}| \le \frac{100 k |A|}{\beta}$. Finally, for every vertex $v$, we have by \Cref{cor:mixing time for sets} that $Q_k^v(\overline{X}) \le \frac{400}{n} |\overline{X}| \le \frac{10^5 k |A|}{\beta n}$, as desired. 
\end{proof}

\noindent For every $t \le \frac{\beta^2n}{10^6 k^2}$, let $A_t \coloneqq \{X^v_j\}_{j\leq tk}$ be the set of vertices visited by the random walk in the first $tk$ steps, and let $Z_{t}$ be the set of vertices $u$ for which $\alpha_u(A_{t-1}) \ge 1-\beta-\frac{200 k^2}{\Delta(G)}$.

Let us now show by induction on $i$ that with probability at least $(1-2\beta)^{i}$ our random walk is self-avoiding after $ik$ steps and moreover, it ends in a vertex of $Z_{i}$. 
By setting $i=\frac{\beta^2n}{10^6 k^2}$ we can then complete our proof of Theorem~\ref{thm:self-avoiding}. As $(1-2\beta) \geq e^{-10\beta}$ for small $\beta>0$, we get $(1-2\beta)^{\frac{\beta^2n}{10^6 k^2}}\geq e^{-\frac{\beta^3n}{10^5 k^2}}$.

To that end, suppose that with probability at least $(1-2\beta)^{i-1}$  our random walk is self-avoiding after $(i-1)k$ steps and moreover, it ends in a vertex $u \in Z_{i-1}$ i.e., $X^v_{(i-1)k}= \{u\}$.

Now, we claim that the probability that our random walk is self-avoiding in the next $k$ steps, avoids $\{X^v_j\}_{j\leq (i-1)k}$ and satisfies that $X^v_{ik}\in Z_i$ is at least 
\begin{align*}
    1- \frac{\beta}{10}-\Big(1-\alpha_u(A_{i-2})\Big)-\frac{200 k^2}{\Delta(G)}\geq 1-\frac{\beta}{10}-\left(\beta + \frac{200 k^2}{\Delta(G)}\right) - \frac{200k^2}{\Delta(G)} \ge 1-2\beta.
\end{align*}
Indeed, by Lemma~\ref{lem:good set}, the probability that $X^v_{ik}\in Z_i$ is at least $1-\frac{\beta}{10}$ and the probability that the random walk $\{X^v_{(i-1)k+1},X^v_{(i-1)k+2}, \ldots, X^v_{ik}\}$ does not avoid $A_{i-2}\cup \{X^v_{(i-2)k+1},X^v_{(i-2)k+2},\ldots, X^v_{(i-1)k}\} = A_{i-1}$ is at most $1-\alpha_u(A_{i-2})$ (which is at most $\beta + \frac{200 k^2}{\Delta(G)}$ since $u \in Z_{i-1}$), 
and by Claim~\ref{selfavoidingkwalk}, the probability that the random walk $\{X^v_{(i-1)k+1},X^v_{(i-1)k+2}\ldots, X^v_{ik}\}$ is not self-avoiding is at most $\frac{200 k^2}{\Delta(G)}$. Putting all of this together and using that $\Delta(G) \ge \frac{10^3 k^2}{\beta}$, the above inequalities hold, as desired. This completes the proof of Theorem~\ref{thm:self-avoiding}.
\end{proof}

\subsection{Putting everything together}
\label{sec:puttingeverythingtogether}

\begin{proof}[Proof of Theorem~\ref{thm:main}]
Let $G$ be a graph on $n$ vertices for $n$ large enough, and with average degree $d(G) \coloneqq d \ge \log^8 n$. First we use Lemma~\ref{findingalmostregexpander} to find a $100$-almost-regular bipartite subgraph $G'$ on $n'$ vertices with average degree at least $\frac{d}{600 \log n}$ which is a $\frac{1}{10 \log n}$-expander. Now, by Corollary \ref{cor:mixing time for sets}, we have that $G'$ has mixing time at most $k := 10^{10} \log ^2 n \log n'$. Let $\beta = 10^{-28}$.

Consider now a random walk $R = \{X^{v_0}_j\}_{j \le t}$ starting at an arbitrary vertex $v_0 \in V(G')$ and of length $t \coloneqq \frac{\beta^2 n'}{10^6 k}$.
Let $\mathcal{E}_1$ be the event that $R$ is self-avoiding, and let $\mathcal{E}_2$ be the event that there is an edge between the first $t/4$ and last $t/4$ vertices of $R$. 

As we have several parameters, we now collect several simple inequalities which hold between them, and which we use to complete our proof. Note first that since the average degree in $G'$ is at least $\frac{d}{600\log n}$ and $G'$ is $100$-almost-regular, we have $\delta(G')\geq \frac{d}{10^5\log n}\geq \frac{\log^7n}{10^5}$. Note further that trivially $n'\geq \delta(G')\geq \frac{\log^7n}{10^5}$, and also that $\delta(G')\gg k^2=O(\log^6n)$.

Now, by Theorem~\ref{thm:self-avoiding} the event $\mathcal{E}_1$ occurs with probability at least $e^{-\frac{\beta^3 n'}{10^5 k^2}} = e^{-\frac{10 \beta t}{k}}$.
Now we want to show that $\frac{t}{4} \ge \max \{ \frac{10^{17}k n'}{\delta(G')}, 10^{25} k \log n' \}$, so that we can apply Lemma \ref{lem:many edges} to obtain that with probability at least $1-e^{-\frac{t/4}{10^{24}k}}$ the event $\mathcal{E}_2$ occurs. Indeed, the first inequality follows from the fact that $t=\Theta(\frac{n'}{k})$ and $\delta(G')\gg k^2$. The second inequality follows from $k^2\log n'=o(n')$. To see why $k^2\log n'=o(n')$ holds, we have two simple cases. If $n'\geq \log^8n$ then this trivially holds as $k^2=O(\log^6 n)$, and otherwise $\log n'=O(\log\log n)$, so again we are done because  $n' = \Omega(\log^7 n)$.

Finally, for each $i\in[k]$, let $W_i$ be the random walk starting at the $(\frac{t}{4}+i)$-th step of the random walk $R$ and finishing at step $\frac{3t}{4}$ of $R$. For each $W_i$, we will consider the set $W_i(k)$, and show that it spans many edges with high probability. Again, we can easily check that $\frac{t/2-i}{2}\geq \frac{t}{5} \geq \max \{ \frac{10^{17}k n'}{\delta(G')}, 10^{25} k \log n' \}$, so by Lemma~\ref{lem:many edges} for every $i\in[k]$ we have that $e(W_i(k)) \ge \frac{(t/5)^2 \delta(G')}{10^{32} k^2 n'} > \frac{2t}{k}$ with probability at least $1-e^{-\frac{t/5}{10^{24}k}}$, as we can split the random walk $W_i$ into two random walks of length at least $t/5$. Let $\mathcal{E}_3$ be the event that for all $i \in [k]$, $e(W_i(k)) \ge \frac{2t}{k}$. Since $k=O(\log^3 n)$ and $\frac{t}{k}=\Omega(\log n)$, by a union bound,
$\mathcal{E}_3$ occurs with probability at least  $1-ke^{-\frac{t}{10^{25}k}}\geq 1-e^{-\frac{t}{10^{26}k}}$.

Since $\beta = 10^{-28}$, we have $\mathbb{P}(\mathcal{E}_1 \cap \mathcal{E}_2 \cap \mathcal{E}_3)\geq  \mathbb P(\mathcal{E}_1)-\mathbb P(\overline{\mathcal{E}_2})-\mathbb P(\overline{\mathcal{E}_3})\geq
e^{-\frac{10 \beta t}{k}}-e^{-\frac{t}{10^{26}k}}-e^{-\frac{t}{10^{25}k}}>0$. Moreover, the event $\mathcal{E}_1 \cap \mathcal{E}_2 \cap \mathcal{E}_3$ implies the existence of a cycle of length $t$ with at least $t$ chords, since if the random walk $R$ is self-avoiding, then the edges spanned by the sets $W_i(k)$ are mutually disjoint for $i \in [k]$. This completes our proof of Theorem~\ref{thm:main}.
\end{proof}

\section*{Concluding remarks}
In this paper we have shown that every $n$-vertex graph with $\Omega(n\log^8 n )$ edges contains a cycle $C$ with at least $|C|$ chords. Although this is a significant improvement upon the previous bound~\cite{chen1996proof} of $\Theta(n^{3/2})$ edges, we believe that the truth is closer to $\Theta(n)$. It would be interesting to show an upper bound of this order (which would be optimal), or to show any lower bound which is super-linear.

Another avenue towards understanding this problem is to consider the following closely related question. What is the largest $t=t(e,n)$ so that any $n$-vertex graph $G$ with $e = e(n)$ edges is guaranteed to contain some cycle $C$ with at least $t |C|$ chords? Let us note that our proof gives $t(e, n)=\Omega\left(\frac{e}{n\log ^7n}\right)$ for $e=\Omega(n\log^8n)$, and that the question from the previous paragraph is whether $t(e,n)\geq 1$ when $e\geq cn$ for a large enough absolute constant $c$.

Let us note that we did not make an attempt to improve the power of the logarithmic factor or the used absolute constants in our result, in order to keep the presentation clean. We expect that one can save a few logarithmic factors by being more careful, but new ideas are certainly required to push the bound very close to $\Theta(n)$, even if we assume the original graph is almost-regular and expanding. Roughly speaking, the reason is that we can only guarantee that the random walk is self-avoiding up to length $O\left(\frac{n}{k}\right)$, where $k$ is the mixing time of the graph (which is at least of order $\log^2 n$ in our proof). Now, if we assume the set of vertices in the random walk behaves like a random set of vertices of size $\Theta\left(\frac{n}{\log^2n}\right)$, then the expected number of edges spanned by the set is $ \Theta\left(\frac{e}{n^2} (\frac{n}{\log^2n})^2\right) = \Theta\left(\frac{e}{\log ^4 n}\right)$, which is at least $\Theta\left(\frac{n}{\log^2n}\right)$ only when we have $e=\Omega(n \log^2n)$ edges in our graph. Additional logarithmic factors are used in our proof for cleaning the graph to find an almost-regular expander in it and because the random walk is not exactly a random set.

\vspace{0.3cm}
\noindent
{\bf Acknowledgements.}
We would like to thank Zachary Hunter
 for helpful comments and for carefully reading our paper.
 
\bibliographystyle{plain}

\begin{thebibliography}{10}

\bibitem{MR539938}
B\'{e}la Bollob\'{a}s.
\newblock Nested cycles in graphs.
\newblock In {\em Probl\`emes combinatoires et th\'{e}orie des graphes
  ({C}olloq. {I}nternat. {CNRS}, {U}niv. {O}rsay, {O}rsay, 1976)}, volume 260
  of {\em Colloq. Internat. CNRS}, pages 49--50. CNRS, Paris, 1978.

\bibitem{bucic2022cycles}
Matija Buci{\'c}, Lior Gishboliner, and Benny Sudakov.
\newblock Cycles of many lengths in {H}amiltonian graphs.
\newblock In {\em Forum of Mathematics, Sigma}, volume~10, page e70. Cambridge
  University Press, 2022.

\bibitem{bucic2020nearly}
Matija Buci{\'c}, Matthew Kwan, Alexey Pokrovskiy, Benny Sudakov, Tuan Tran,
  and Adam~Zsolt Wagner.
\newblock Nearly-linear monotone paths in edge-ordered graphs.
\newblock {\em Israel Journal of Mathematics}, 238(2):663--685, 2020.

\bibitem{chen1996proof}
Guantao Chen, Paul Erd{\H{o}}s, and William Staton.
\newblock Proof of a conjecture of {B}ollob\'{a}s on {N}ested {C}ycles.
\newblock {\em Journal of Combinatorial Theory, Series B}, 66(1):38--43, 1996.

\bibitem{erdHos1975problems}
Paul Erd{\H{o}}s.
\newblock Problems and {R}esults in {G}raph {T}heory and {C}ombinatorial
  {A}nalysis.
\newblock {\em Proc. British Combinatorial Conj., 5th}, pages 169--192, 1975.

\bibitem{fernandez2022nested}
Irene Fern{\'a}ndez, Jaehoon Kim, Younjin Kim, and Hong Liu.
\newblock Nested cycles with no geometric crossings.
\newblock {\em Proceedings of the American Mathematical Society, Series B},
  9(3):22--32, 2022.

\bibitem{fernandez2023build}
Irene~Gil Fern{\'a}ndez and Hong Liu.
\newblock How to build a pillar: A proof of {T}homassen's conjecture.
\newblock {\em Journal of Combinatorial Theory, Series B}, 162:13--33, 2023.

\bibitem{janzer2022resolution}
Oliver Janzer and Benny Sudakov.
\newblock Resolution of the {E}rd{\H{o}}s-{S}auer problem on regular subgraphs.
\newblock {\em arXiv preprint arXiv:2204.12455}, 2022.

\bibitem{jiang2021rainbow}
Tao Jiang, Shoham Letzter, Abhishek Methuku, and Liana Yepremyan.
\newblock Rainbow clique subdivisions and blow-ups.
\newblock {\em arXiv preprint arXiv:2108.08814}, 2021.

\bibitem{liu2023solution}
Hong Liu and Richard Montgomery.
\newblock A solution to {E}rd{\H{o}}s and {H}ajnal’s odd cycle problem.
\newblock {\em Journal of the American Mathematical Society}, 2023.

\bibitem{lovasz1993random}
L{\'a}szl{\'o} Lov{\'a}sz.
\newblock Random walks on graphs: A survey.
\newblock {\em Combinatorics, Paul Erd{\H{o}}s is {E}ighty}, 2:1--46, 1993.

\bibitem{pak2002mixing}
Igor Pak.
\newblock Mixing time and long paths in graphs.
\newblock In {\em Symposium on Discrete Algorithms: Proceedings of the
  thirteenth annual ACM-SIAM symposium on Discrete algorithms}, volume~6, pages
  321--328, 2002.

\bibitem{thomassen1983girth}
Carsten Thomassen.
\newblock Girth in graphs.
\newblock {\em Journal of Combinatorial Theory, Series B}, 35(2):129--141,
  1983.

\bibitem{thomassen1989configurations}
Carsten Thomassen.
\newblock Configurations in graphs of large minimum degree, connectivity, or
  chromatic number.
\newblock {\em Annals of the New York Academy of Sciences}, 555(1):402--412,
  1989.

\end{thebibliography}

\end{document}